\documentclass[11pt,reqno,sumlimits]{amsart}
\usepackage{amsfonts, amsmath, amscd, amssymb, euscript, amsthm, array, booktabs, color, dcolumn, shortvrb, tabularx, units, url, mathrsfs, enumitem, tikz-cd, nameref, bm}
\usepackage[pdfstartview=FitH]{hyperref}
\usepackage[all]{xy}
\usepackage{kpfonts}
\setlist[enumerate,1]{label={(\alph*)}}
\normalsize
\makeatletter
\newcommand{\addresseshere}{\enddoc@text\let\enddoc@text\relax}
\makeatother

\DeclareMathOperator{\bun}{Bun}

\DeclareMathOperator{\op}{op}

\DeclareMathOperator{\re}{Re}
\DeclareMathOperator{\CCS}{CCS}
\DeclareMathOperator{\Aut}{Aut}

\DeclareMathOperator{\dr}{dR}

\DeclareMathOperator{\ind}{ind}

\DeclareMathOperator{\CS}{CS}
\DeclareMathOperator{\odd}{odd}
\DeclareMathOperator{\even}{even}
\DeclareMathOperator{\im}{Im}
\DeclareMathOperator{\id}{id}
\DeclareMathOperator{\ch}{ch}
\DeclareMathOperator{\End}{End}

\DeclareMathOperator{\rk}{rank}
\DeclareMathOperator{\str}{str}

\begin{document}
\setlength{\baselineskip}{1.5\baselineskip}
\theoremstyle{definition}
\newtheorem{coro}{Corollary}
\newtheorem*{note}{Note}
\newtheorem{thm}{Theorem}
\newtheorem*{ax}{Axiom}
\newtheorem{defi}{Definition}
\newtheorem{lemma}{Lemma}
\newtheorem*{claim}{Claim}
\newtheorem{exam}{Example}
\newtheorem{prop}{Proposition}
\newtheorem{remark}{Remark}
\newcommand{\hto}{\hookrightarrow}
\newcommand{\wt}[1]{{\widetilde{#1}}}
\newcommand{\ov}[1]{{\overline{#1}}}
\newcommand{\un}[1]{{\underline{#1}}}
\newcommand{\wh}[1]{{\widehat{#1}}}
\newcommand{\wo}{\mathbin{\wh{\otimes}}}
\newcommand{\deff}[1]{{\bf\emph{#1}}}
\newcommand{\boo}[1]{\boldsymbol{#1}}
\newcommand{\abs}[1]{\lvert#1\rvert}
\newcommand{\norm}[1]{\lVert#1\rVert}
\newcommand{\inner}[1]{\langle#1\rangle}
\newcommand{\poisson}[1]{\{#1\}}
\newcommand{\biginner}[1]{\Big\langle#1\Big\rangle}
\newcommand{\set}[1]{\{#1\}}
\newcommand{\Bigset}[1]{\Big\{#1\Big\}}
\newcommand{\BBigset}[1]{\bigg\{#1\bigg\}}
\newcommand{\dis}[1]{$\displaystyle#1$}
\newcommand{\KER}{\bm{\mathrm{Ker}}}
\newcommand{\EE}{\bm{\mathrm{E}}}
\newcommand{\FF}{\bm{\mathrm{F}}}
\newcommand{\JJ}{\bm{\mathrm{J}}}
\newcommand{\HH}{\bm{\mathrm{H}}}
\newcommand{\KK}{\bm{\mathrm{K}}}
\newcommand{\LL}{\bm{\mathrm{L}}}
\newcommand{\VV}{\bm{\mathrm{V}}}
\newcommand{\WW}{\bm{\mathrm{W}}}
\newcommand{\CC}{\bm{\mathrm{C}}}
\newcommand{\SSS}{\bm{\mathrm{S}}}
\newcommand{\R}{\mathbb{R}}
\newcommand{\N}{\mathbb{N}}
\newcommand{\Z}{\mathbb{Z}}
\newcommand{\Q}{\mathbb{Q}}
\newcommand{\E}{\mathcal{E}}
\newcommand{\T}{\mathcal{T}}
\newcommand{\G}{\mathcal{G}}
\newcommand{\F}{\mathcal{F}}
\newcommand{\I}{\mathcal{I}}
\newcommand{\V}{\mathcal{V}}
\newcommand{\W}{\mathcal{W}}
\newcommand{\C}{\mathbb{C}}
\newcommand{\A}{\mathcal{A}}
\newcommand{\HHH}{\mathcal{H}}
\newcommand{\PP}{\mathcal{P}}
\newcommand{\K}{\mathcal{K}}
\newcommand{\RRR}{\mathscr{R}}
\newcommand{\DDD}{\mathscr{D}}
\newcommand{\KKK}{\mathscr{K}}
\newcommand{\LLL}{\mathscr{L}}
\newcommand{\JJJ}{\mathscr{J}}
\newcommand{\e}{\mathscr{E}}
\newcommand{\hh}{\mathscr{H}}
\newcommand{\kk}{\mathscr{K}}
\newcommand{\jj}{\mathscr{J}}
\newcommand{\w}{\mathscr{W}}
\newcommand{\f}{\mathscr{F}}
\newcommand{\g}{\mathscr{G}}
\newcommand{\so}{\mathfrak{so}}
\newcommand{\gl}{\mathfrak{gl}}
\newcommand{\aaa}{\mathbb{A}}
\newcommand{\bbb}{\mathbb{B}}
\newcommand{\ttt}{\mathbb{T}}
\newcommand{\DD}{\mathsf{D}}
\newcommand{\ff}{\mathsf{F}}
\newcommand{\sss}{\mathbb{S}}
\newcommand{\cdd}[1]{\[\begin{CD}#1\end{CD}\]}
\numberwithin{equation}{section}
\normalsize
\title[RRG II]{Local index theory and the Riemann-Roch-Grothendieck theorem for complex flat vector bundles II}
\author{Man-Ho Ho}
\email{homanho@bu.edu}
\address{Hong Kong}
\maketitle
\nocite{*}
\begin{center}
\emph{\small Dedicated to the memory of Connie}
\end{center}
\begin{abstract}
In this paper, we prove the real part of the Riemann--Roch--Grothendieck theorem for complex flat vector bundles at the differential form level.
\end{abstract}
\section{Introduction}

Let $\pi:X\to B$ be a submersion with closed fibers $Z$, equipped with a Riemannian structure $\bm{\pi}=(T^HX, g^{T^VX})$. The Riemann--Roch--Grothendieck (RRG) theorem for complex flat vector bundles \cite[Theorem 1.1]{MZ08} states that for any complex flat vector bundle $(F, \nabla^F)$ over $X$,
\begin{equation}\label{eq 1.1}
\CCS(H(Z, F|_Z), \nabla^{H(Z, F|_Z)})=\int_{X/B}e(T^VX)\cup\CCS(F, \nabla^F)
\end{equation}
in $H^{\odd}(B; \C/\Q)$, where $\CCS(F, \nabla^F)$ and $(H(Z, F|_Z), \nabla^{H(Z, F|_Z)})$ are the Cheeger--Chern--Simons class and the cohomology bundle of $(F, \nabla^F)$, respectively.

Bismut--Lott prove a refinement of the imaginary part of (\ref{eq 1.1}) at the differential form level \cite[Theorem 3.23]{BL95} by constructing the analytic torsion form $T$ \cite[Definition 3.22]{BL95}. Bismut then proves the real part of (\ref{eq 1.1}), namely,
\begin{equation}\label{eq 1.2}
\re(\CCS(H(Z, F|_Z), \nabla^{H(Z, F|_Z)}))=\int_{X/B}e(T^VX)\cup\re(\CCS(F, \nabla^F))
\end{equation}
in $H^{\odd}(B; \R/\Q)$ under the assumption that the fibers are fiberwise orientable \cite[Theorem 3.2]{B05}. Ma--Zhang prove (\ref{eq 1.1}) by giving another proof of the imaginary part of (\ref{eq 1.1}) and prove (\ref{eq 1.2}) in full generality.

The strategy of the proof of (\ref{eq 1.2}) by Ma--Zhang is roughly as follows. They first prove that
\begin{equation}\label{eq 1.3}
\begin{split}
&\re(\CCS(H(Z, F|_Z), \nabla^{H(Z, F|_Z)}))-\re(\CCS(H(Z, \C^\ell|_Z)))\\
&=\int_{X/B}e(T^VX)\cup\re(\CCS(F, \nabla^F))
\end{split}
\end{equation}
in $H^{\odd}(B; \R/\Q)$ \cite[(3.98)]{MZ08}, where $\ell=\rk(F)$. Then they show that \cite[p.613-614]{MZ08}
\begin{equation}\label{eq 1.4}
\re(\CCS(H(Z, \C^\ell|_Z)))=0
\end{equation}
by considering the parity of $\dim(Z)$ separately. For $\dim(Z)$ odd, they prove (\ref{eq 1.4}) by an application of Poincar\'e duality; for $\dim(Z)$ even, they prove (\ref{eq 1.4}) by applying a result of Bismut \cite[Theorem 3.12]{B05}, which is essentially an application of the special case of (\ref{eq 1.2}) proved by Bismut.

The main result of this paper is a proof of (\ref{eq 1.2}) at the differential form level. This gives a new proof of (\ref{eq 1.2}).

\begin{thm}\label{thm 1}($=$ Theorem \ref{thm 2})
Let $\pi:X\to B$ be a submersion with closed fibers $Z$, equipped with a Riemannian structure $\bm{\pi}$. For any complex flat vector bundle $(F, \nabla^F)$ of rank $\ell$ over $X$, there exist sufficiently large $k, N\in\N$, an isometric isomorphism $\alpha:(kF, kg^F)\to(\C^{k\ell}, g^{k\ell})$ and a $\Z_2$-graded isometric isomorphism
$$h:(kH(Z, F|_Z)\oplus\wh{\C}^N, kg^{H(Z, F|_Z)}\oplus\wh{g}^N)\to(\C^{km}\oplus\wh{\C}^N, g^{km}\oplus\wh{g}^N)$$
such that
\begin{equation}\label{eq 1.5}
\CS\big(k\nabla^{H(Z, F|_Z), u}\oplus\wh{d}^N, h^*(d^{km}\oplus\wh{d}^N)\big)\equiv\int_{X/B}e(\nabla^{T^VX})\wedge\CS(k\nabla^{F, u}, \alpha^*d^{k\ell}),
\end{equation}
where $m=\rk(H(Z, F|_Z))$.
\end{thm}

In the final step of the proof of Theorem \ref{thm 1}, we make use of a result by Bismut \cite[Theorem 3.7]{B05} (see also Zhang \cite[Proposition 2.3]{Z04}).

Although the left-hand side of (\ref{eq 1.5}) does not look like a differential form representative of $\re(\CCS(H(Z, F|_Z), \nabla^{H(Z, F|_Z)}))$, in Theorem \ref{thm 3} we justify the claim that (\ref{eq 1.5}) is a refinement of (\ref{eq 1.2}) at the differential form level. On the other hand, by taking $(F, \nabla^F)$ to be the trivial bundle $(\C^\ell, d^\ell)$ in Theorem \ref{thm 1} and noting that a differential form representative of $\CCS(\C^\ell, d^\ell)$ is given by an odd Chern character form, which implies $\CCS(\C^\ell, d^\ell)=0$ in $H^{\odd}(X; \C/\Q)$, we obtain a unified proof of (\ref{eq 1.4}) at the differential form level regardless of the parity of $\dim(Z)$.

The proof of Theorem \ref{thm 1} is similar to that of \cite[Theorem 1]{H20} in spirit. However, there are two notable differences. The first difference is the use of a perturbed generalized twisted de Rham operator in the proof of Theorem \ref{thm 1}, in contrast to the perturbed twisted spin$^c$ Dirac operator used in \cite[Theorem 1]{H20}. The former allows us to treat the even and odd dimensional fiber cases simultaneously, while the latter only allows us to treat the even dimensional fiber case. The second difference is a refinement of the proof of \cite[Theorem 1]{H20}, namely, we take a more suitable pair of unitary connections on the ``index bundle" over $B\times[0, 1]$ when defining a certain Chern--Simons form. This is the main reason why we can directly consider complex flat vector bundles instead of $\Z_2$-graded ones of virtual rank zero, as we did in \cite{H20}.

While Theorem \ref{thm 1} is a generalization of \cite[Theorem 1]{H20}, we would like to emphasize that \cite[Theorem 1]{H20} is much weaker than Theorem \ref{thm 1}. First, \cite[Theorem 1]{H20} only holds for complex flat vector bundles that are $\Z_2$-graded and have virtual rank zero. Second, and more importantly, we have come to realize the difficulty of proving (\ref{eq 1.4}) independently when we try to obtain (\ref{eq 1.2}) from (\ref{eq 1.3}). Although (\ref{eq 1.4}) is now a consequence of Theorem \ref{thm 1}, we still regard proving (\ref{eq 1.4}) independently as a difficult problem, and the proof of (\ref{eq 1.4}) given by Ma--Zhang and Bismut is more intricate than we expect.

\section*{Acknowledgement}

The author is indebted to Steve Rosenberg for his comments and suggestions for this paper. Some of the ideas in this paper are inspired during a visit to the Institute of Mathematics of the Academia Sinica, Taiwan in July, 2019. The author would like to thank Jih-Hsin Cheng for the invitation, Yi-Sheng Wang for many enlightening discussions, and the Academia Sinica for the hospitality.

The author would like to dedicate this paper to his beloved chihuahua, Connie, who was a cherished member of his family. We are lucky to have each other for many years. Her memory will always be with us.

\section{Background Material}\label{s 2}

\subsection{Notations and conventions}\label{s 2.1}

In this subsection, we fix the notations and conventions used in this paper.

In this paper $X$ and $B$ are closed manifolds and $I$ is the closed interval $[0, 1]$. Given a manifold $X$, define $\wt{X}=X\times I$. Given $t\in I$, define a map $i_{X, t}:X\to\wt{X}$ by $i_{X, t}(x)=(x, t)$. Denote by $p_X:\wt{X}\to X$ the standard projection map. For $k\geq 0$, denote by $\Omega^k_\Q(X, \C)$ the set of all complex-valued closed $k$-forms on $X$ with periods in $\Q$, and write $\Omega^k_\Q(X)$ for $\Omega^k_\Q(X, \R)$. For any differential forms $\omega$ and $\eta$, we write $\omega\equiv\eta$ if $\omega-\eta\in\im(d)$.

To shorten the notations and terminologies, we will work with the category $\bun_\nabla(X)$ whose objects $\EE=(E, g^E, \nabla^E)$, called geometric bundles, are Hermitian bundles and unitary connections and whose morphisms $\alpha:\EE\to\FF$ are isometric isomorphisms $\alpha:(E, g^E)\to(F, g^F)$, i.e. $g^E=\alpha^*g^F$.

Here are some terminologies and notations regarding geometric bundles.
\begin{itemize}
  \item If $\alpha:\EE\to\FF$ is a morphism, then $\alpha^*\nabla^F:=\alpha^{-1}\circ\nabla^F\circ\alpha$ is a unitary connection on $(E, g^E)$.
  \item A geometric bundle $\EE=(E, g^E, \nabla^E)$ is said to be $\Z_2$-graded if $E\to X$, $g^E$ and $\nabla^E$ are $\Z_2$-graded.
  \item For a $\Z_2$-graded geometric bundle $\EE$, denote by $\EE^{\op}$ the $\Z_2$-graded geometric bundle whose $\Z_2$-graded grading is opposite to that of $\EE$.
  \item A $\Z_2$-graded geometric bundle $\EE$ is said to be balanced if $E^+=E^-$, $g^{E, +}=g^{E, -}$ and $\nabla^{E, +}=\nabla^{E, -}$.
  \item For any $\ell\in\N$, write $\CC^\ell=(\C^\ell, g^\ell, d^\ell)$ for the geometric bundle that consists of the trivial bundle of rank $\ell$ with the standard Hermitian metric and the standard trivial connection.
  \item For a geometric bundle $\EE=(E, g^E, \nabla^E)$, write $\wh{\EE}=(\wh{E}, g^{\wh{E}}, \nabla^{\wh{E}})$ for the balanced $\Z_2$-graded geometric bundle defined by $\wh{E}^\pm=E$, and similarly for $g^{\wh{E}}$ and $\nabla^{\wh{E}}$. In the case that $\EE=\CC^\ell$, we write $\wh{\CC}^\ell=(\wh{\C}^\ell, \wh{g}^\ell, \wh{d}^\ell)$ for $\wh{\EE}$.
\end{itemize}

Let $(E, g^E)$ be a Hermitian bundle with a connection $\nabla^E$ on $E\to X$ that is not assumed to be unitary with respect to $g^E$. Define
\begin{equation}\label{eq 2.1}
\omega(E, g^E):=(g^E)^{-1}(\nabla^Eg^E)\in\Omega^1(X, \End(E)).
\end{equation}
By \cite[p.62]{BZ92}, the connection $\nabla^{E, u}$ on $E\to X$ defined
\begin{equation}\label{eq 2.2}
\nabla^{E, u}=\nabla^E+\frac{1}{2}\omega(E, g^E)
\end{equation}
is unitary with respect to $g^E$. Write $\EE^u$ for $(E, g^E, \nabla^{E, u})$.

In this paper, we will constantly make use of Chern character form and Chern--Simons form. We refer the readers to \cite[\S2.2]{H23} for the details of the constructions and properties of these forms.

\subsection{Local index theory for perturbed generalized twisted de Rham operators}\label{s 2.2}

In this subsection, we review the local family index theorem for perturbed generalized twisted de Rham operators without the kernel bundle assumption. We refer to \cite[Chapter 10]{BGV}, \cite[\S III]{BL95} and \cite[\S7.12]{FL10} for the details.

Let $\pi:X\to B$ be a submersion with closed fibers $Z$ of dimension $n$. Denote by $T^VX\to X$ the vertical tangent bundle. Recall from \cite[p.918]{FL10} that a Riemannian structure $\bm{\pi}=(T^HX, g^{T^VX})$ on $\pi:X\to B$ consists of a horizontal distribution $T^HX\to X$, i.e. $TX=T^VX\oplus T^HX$, and a metric $g^{T^VX}$ on $T^VX\to X$. Put a Riemannian metric $g^{TB}$ on $TB\to B$. Define a metric $g^{TX}$ on $TX\to X$ by
$$g^{TX}=g^{T^VX}\oplus\pi^*g^{TB}.$$
Denote by $\nabla^{TX}$ and $\nabla^{TB}$ the Levi-Civita connections on $TX\to X$ and $TB\to X$ associated to $g^{TX}$ and $g^{TB}$, respectively. Denote by $P^{T^VX}:TX\to T^VX$ the projection map. Then $\nabla^{T^VX}:=P^{T^VX}\circ\nabla^{TX}\circ P^{T^VX}$ is a metric connection on the Riemannian bundle $(T^VX, g^{T^VX})$.

Define a connection $\wt{\nabla}^{TX}$ on $TX\to X$ by
$$\wt{\nabla}^{TX}=\nabla^{T^VX}\oplus\pi^*\nabla^{TB}.$$
Let $S=\nabla^{TX}-\wt{\nabla}^{TX}\in\Omega^1(X, \End(TX))$. By \cite[Theorem 1.9]{B86} the $(3, 0)$ tensor $g^{TX}(S(\cdot)\cdot, \cdot)$ depends only on the Riemannian structure $\bm{\pi}=(T^HX, g^{T^VX})$. Let $\set{e_1, \ldots, e_n}$ be a local orthonormal frame for $T^VX\to X$. For any $U\in\Gamma(B, TB)$, denote by $U^H\in\Gamma(X, T^HX)$ its horizontal lift. Define a horizontal one-form $k$ on $X$ by
\begin{equation}\label{eq 2.3}
k(U^H)=-\sum_{k=1}^ng^{TX}(S(e_k)e_k, U^H).
\end{equation}
For any $U, V\in\Gamma(B, TB)$,
\begin{equation}\label{eq 2.4}
T(U, V):=-P^{T^VX}[U^H, V^H]
\end{equation}
is a horizontal two-form with values in $T^VX$, and is called the curvature of $\pi:X\to B$.

The complexified exterior bundle $\Lambda(T^VX)^*\to X$ has a natural $\Z_2$-grading and is a Clifford module with Clifford multiplication $c(Y)=\varepsilon(Y)-i(Y)$, where $Y\in\Gamma(X, T^VX)$, and $\varepsilon$ and $i$ are the exterior and interior multiplications, respectively. Denote by $g^{\Lambda(T^VX)^*}$ and $\nabla^{\Lambda(T^VX)^*}$ the extensions of $g^{T^VX}$ and $\nabla^{T^VX}$ to $\Lambda(T^VX)^*\to X$, respectively.

Let $\EE^u=(E, g^E, \nabla^{E, u})$. The generalized twisted de Rham operator $\DD^{\Lambda\otimes E}$ associated to $\EE^u$ and $\bm{\pi}$ is defined to be
$$\DD^{\Lambda\otimes E}=\sum_{k=1}^nc(e_k)\nabla^{\Lambda(T^VX)^*\otimes E, u}_{e_k},$$
where $\nabla^{\Lambda(T^VX)^*\otimes E, u}$ is the tensor product of $\nabla^{\Lambda(T^VX)^*}$ and $\nabla^{E, u}$. It acts on $\Gamma(X, \Lambda(T^VX)^*\otimes E)$ and is odd self-adjoint, i.e. $\DD^{\Lambda\otimes E}_-=(\DD^{\Lambda\otimes E}_+)^*$. Set $\wh{c}(Y)=\varepsilon(Y)+i(Y)$. Define a map $V:\Gamma(X, \Lambda(T^VX)^*\otimes E)\to\Gamma(X, \Lambda(T^VX)^*\otimes E)$ by
\begin{equation}\label{eq 2.5}
V=-\frac{1}{2}\sum_{k=1}^n\wh{c}(e_k)\omega(E, g^E)(e_k).
\end{equation}
Note that $V$ is odd self-adjoint and anti-commutes with the $c(X)$'s. Thus
$$\DD^{\Lambda\otimes E; V}:=\DD^{\Lambda\otimes E}+V$$
is also odd self-adjoint.

Define an infinite rank $\Z_2$-graded complex vector bundle $\pi^\Lambda_*E\to B$ whose fiber over $b\in B$ is
$$(\pi^\Lambda_*E)_b=\Gamma(Z_b, (\Lambda(T^VX)^*\otimes E)|_{Z_b}).$$
Note that $\Omega(X, E)\cong\Omega(B, \pi^\Lambda_*E)$. Denote by $\ast$ the fiberwise Hodge star operator associated to $g^{T^VX}$, and extend it from $\Gamma(X, \Lambda(T^VX)^*)$ to $\Gamma(X, \Lambda(T^VX)^*\otimes E)\cong\Gamma(B, \pi^\Lambda_*E)$. Define an $L^2$-metric on $\pi^\Lambda_*E\to B$ by
$$g^{\pi^\Lambda_*E}(s_1, s_2)(b)=\int_{Z_b}g^E(s_1(b)\wedge\ast s_2(b)).$$

The connection on $\pi^\Lambda_*E\to B$ defined by
$$\nabla^{\pi^\Lambda_*E}_Us=\nabla^{\Lambda(T^VX)^*\otimes E}_{U^H}s,$$
where $s\in\Gamma(B, \pi^\Lambda_*E)$ and $U\in\Gamma(B, TB)$, is $\Z_2$-graded. Then
$$\nabla^{\pi^\Lambda_*E, u}:=\nabla^{\pi^\Lambda_*E}+\frac{1}{2}k,$$
where $k$ is given by (\ref{eq 2.3}), is a $\Z_2$-graded unitary connection on the Hermitian bundle $(\pi^\Lambda_*E, g^{\pi^\Lambda_*E})$. Write $\pi^\Lambda_*\EE^u=(\pi^\Lambda_*E, g^{\pi^\Lambda_*E}, \nabla^{\pi^\Lambda_*E, u})$.

Define the rescaled Bismut superconnection $\pi^\Lambda_*E\to B$ by
$$\bbb^{\Lambda, E; V}_t=\sqrt{t}\DD^{\Lambda\otimes E; V}+\nabla^{\pi^\Lambda_*E, u}-\frac{c(T)}{4\sqrt{t}},$$
where $T$ is given by (\ref{eq 2.4}). By \cite[(3.76)]{BL95}, we have
\begin{equation}\label{eq 2.6}
\lim_{t\to 0}\ch(\bbb^{\Lambda, E; V}_t)=\int_{X/B}e(\nabla^{T^VX})\wedge\ch(\nabla^{E, u}).
\end{equation}
\begin{remark}
For any given complex flat vector bundle $(F, \nabla^F)$, \cite[(3.76)]{BL95} states that
$$\lim_{t\to 0}\ch(\C_t)=\int_{X/B}e(\nabla^{T^VX})\wedge\ch(\nabla^{F, u}),$$
where $\C_t$ is the rescaled superconnection defined by \cite[(3.50)]{BL95}. Under our notations, \cite[(3.59)]{BL95} states that
\begin{equation}\label{eq 2.7}
\C_{4t}=\bigg(\sqrt{t}\sum_{k=1}^nc(e_k)\nabla^{\Lambda(T^VX)^*\otimes F, u}_{e_k}+\nabla^{\pi^\Lambda_*F, u}-\frac{c(T)}{4\sqrt{t}}\bigg)-\frac{\sqrt{t}}{2}\sum_{k=1}^n\wh{c}(e_k)\omega(F, g^F)(e_k).
\end{equation}
The right-hand side of (\ref{eq 2.7}) is exactly $\bbb^{\Lambda, F; V}_t$. Thus (\ref{eq 2.6}) holds for complex flat vector bundles. Since the proof of \cite[(3.76)]{BL95} does not depend on the flatness of $\nabla^F$, (\ref{eq 2.6}) holds.
\end{remark}

By applying \cite[Lemma 2.3]{MF79} to $\DD^{\Lambda\otimes E; V}$, Mi\v s\v cenko--Fomenko \cite[p.96-97]{MF79} (see also \cite[Lemma 7.13]{FL10}) prove that there exist finite rank subbundles $L^\pm\to B$ and complementary closed subbundles $K^\pm\to B$ of $(\pi^\Lambda_*E)^\pm\to B$ such that
\begin{equation}\label{eq 2.8}
(\pi^\Lambda_*E)^+=K^+\oplus L^+,\qquad(\pi^\Lambda_*E)^-=K^-\oplus L^-,
\end{equation}
$\DD^{\Lambda\otimes E; V}_+:(\pi^\Lambda_*E)^+\to(\pi^\Lambda_*E)^-$ is block diagonal as a map with respect to (\ref{eq 2.8}), and $\DD^{\Lambda\otimes E; V}_+|_{K^+}:K^+\to K^-$ is a smooth bundle isomorphism.

Given $L^\pm\to B$ satisfying the above conditions, we say the $\Z_2$-graded complex vector bundle $L\to B$, defined by $L=L^+\oplus L^-$, satisfies the MF property for $\DD^{\Lambda\otimes E; V}$. If $L\to B$ satisfies the MF property for $\DD^{\Lambda\otimes E; V}$, then the analytic index of $[E]\in K(X)$ associated to $\DD^{\Lambda\otimes E; V}$ is defined to be
$$\ind^{a, \Lambda}([E])=[L^+]-[L^-]\in K(B).$$
It is proved in \cite[p.96-97]{MF79} that $\ind^{a, \Lambda}([E])$ is independent of the choice of $L\to B$ satisfying the MF property for $\DD^{\Lambda\otimes E; V}$.
\begin{remark}\label{remark 2}
If $L_0\to B$ and $L_1\to B$ are $\Z_2$-graded complex vector bundles satisfying the MF property for $\DD^{\Lambda\otimes E; V}$, it follows from above and \cite[p.289]{BGV} that there exist complex vector bundles $V_0\to B$ and $V_1\to B$ such that
$$L_0\oplus\wh{V}_0\cong L_1\oplus\wh{V}_1.$$
\end{remark}

Let $g^L$ be the $\Z_2$-graded Hermitian metric on $L\to B$ inherited from $g^{\pi^\Lambda_*E}$. Denote by $P:\pi^\Lambda_*E\to L$ the $\Z_2$-graded projection map with respect to (\ref{eq 2.8}). Then
$$\nabla^L:=P\circ\nabla^{\pi^\Lambda_*E, u}\circ P$$
is a $\Z_2$-graded unitary connection on the Hermitian bundle $(L, g^L)$. Write
$$\LL=(L, g^L, \nabla^L).$$
Henceforth, whenever $\LL$ is a $\Z_2$-graded geometric bundle for which $L\to B$ satisfies the MF property for $\DD^{\Lambda\otimes E; V}$, $g^L$ and $\nabla^L$ are obtained as above unless otherwise specified.

Let $i_-:L^-\to(\pi^\Lambda_*E)^-$ be the inclusion map. For any $z\in\C$, define a map $\wt{\DD}^{\Lambda\otimes E; V}_+(z):(\pi^\Lambda_*E\oplus L^{\op})^+\to(\pi^\Lambda_*E\oplus L^{\op})^-$ by
$$\wt{\DD}^{\Lambda\otimes E; V}_+(z)=\begin{pmatrix} \DD^{\Lambda\otimes E; V}_+ & zi_- \\ zP_+ & 0 \end{pmatrix}.$$
Note that $\wt{\DD}^{\Lambda\otimes E; V}_+(z)$ is invertible for any $z\neq 0$ (cf. \cite[Lemma 7.20]{FL10}). Define a map $\wt{\DD}^{\Lambda\otimes E; V}(z):\pi^\Lambda_*E\oplus L^{\op}\to\pi^\Lambda_*E\oplus L^{\op}$ by
$$\wt{\DD}^{\Lambda\otimes E; V}(z):=\begin{pmatrix} 0 & (\wt{\DD}^{\Lambda\otimes E; V}_+(z))^* \\ \wt{\DD}^{\Lambda\otimes E; V}_+(z) & 0 \end{pmatrix}.$$

Choose and fix $a\in(0, 1)$ and a smooth increasing function $\alpha:[0, \infty)\to I$ satisfying $\alpha(t)=0$ for all $t\leq a$ and $\alpha(t)=1$ for all $t\geq 1$. Define a rescaled Bismut superconnection on $\pi^\Lambda_*E\oplus L^{\op}\to B$ by
\begin{equation}\label{eq 2.9}
\wh{\bbb}^{\Lambda, E; V}_t=\sqrt{t}\wt{\DD}^{\Lambda\otimes E; V}(\alpha(t))+\big(\nabla^{\pi^\Lambda_*E, u}\oplus\nabla^{L, \op}\big)-\frac{c(T)}{4\sqrt{t}}.
\end{equation}
Since $\wt{\DD}^{\Lambda\otimes E; V}(\alpha(t))$ is invertible for any $t\geq 1$, it follows that
\begin{equation}\label{eq 2.10}
\lim_{t\to\infty}\ch(\wh{\bbb}^{\Lambda, E; V}_t)=0.
\end{equation}
On the other hand, for $t\leq a$, $\wh{\bbb}^{\Lambda, E; V}_t$ decouples, i.e.
\begin{equation}\label{eq 2.11}
\wh{\bbb}^{\Lambda, E; V}_t=\bigg(\sqrt{t}\DD^{\Lambda\otimes E; V}+\nabla^{\pi^\Lambda_*E, u}-\frac{c(T)}{4\sqrt{t}}\bigg)\oplus\nabla^{L, \op}=\bbb^{\Lambda, E; V}_t\oplus\nabla^{L, \op}.
\end{equation}
By (\ref{eq 2.6}), we have
\begin{equation}\label{eq 2.12}
\begin{split}
\lim_{t\to 0}\ch(\wh{\bbb}^{\Lambda, E; V}_t)&=\lim_{t\to 0}\ch(\bbb^{\Lambda, E; V}_t)-\ch(\nabla^L)\\
&=\int_{X/B}e(\nabla^{T^VX})\wedge\ch(\nabla^{E, u})-\ch(\nabla^L).
\end{split}
\end{equation}
The Bismut--Cheeger eta form associated to $\wh{\bbb}^{\Lambda, E; V}_t$ is defined to be
\begin{equation}\label{eq 2.13}
\wh{\eta}^\Lambda(\EE^u, \bm{\pi}, \LL)=\frac{1}{\sqrt{\pi}}\int^\infty_0\str\bigg(\frac{d\wh{\bbb}^{\Lambda, E; V}_t}{dt}e^{-\frac{1}{2\pi i}(\wh{\bbb}^{\Lambda, E; V}_t)^2}\bigg)dt.
\end{equation}
By (\ref{eq 2.10}) and (\ref{eq 2.12}), the local family index theorem for $\DD^{\Lambda\otimes E; V}$ with respect to $\LL$ is given by
$$d\wh{\eta}^\Lambda(\EE^u, \bm{\pi}, \LL)=\int_{X/B}e(\nabla^{T^VX})\wedge\ch(\nabla^{E, u})-\ch(\nabla^L).$$

\section{Main Results}\label{s 3}

We prove the main results in this section. We refer the readers to \cite[\S III]{BL95} for the details of the local index theory for twisted de Rham operators, and to \cite[\S2.2 and \S2.3]{H20} for a quick summary. In this section, we will repeatedly make use of the following facts.
\begin{itemize}
  \item For any given complex vector bundle $E\to X$ with a connection $\nabla^E$ satisfying $\ch(\nabla^E)=\rk(E)$, there exists a (sufficiently large) $k\in\N$ such that $kE\cong k\C^{\rk(E)}$ as complex vector bundles (cf. \cite[Remark 1]{H20}). In particular, this applies to any complex flat vector bundle.
  \item Let $\EE$ and $\FF$ be geometric bundles. If there exists a bundle isomorphism $\alpha:E\to F$, there always exists an isometric isomorphism, still denoted by $\alpha:(E, g^E)\to(F, g^F)$, that is uniquely determined by $g^E$ and $g^F$ \cite[(2) of Remark 2.1]{H23}. Thus $\alpha:\EE\to\FF$ is a morphism.
\end{itemize}

\subsection{An equality of odd degree differential forms}\label{s 3.1}

In this subsection, we prove Theorem \ref{thm 1}.

Let $\pi:X\to B$ be a submersion with closed fibers $Z$, equipped with a Riemannian structure $\bm{\pi}$. Given a complex flat vector bundle $(F, \nabla^F)$ over $X$ and a Hermitian metric $g^F$, write $\FF^u=(F, g^F, \nabla^{F, u})$ and
$$\HH(Z, F|_Z)=(H(Z, F|_Z), g^{H(Z, F|_Z)}, \nabla^{H(Z, F|_Z), u}),$$
where $(H(Z, F|_Z), \nabla^{H(Z, F|_Z)})$ is the cohomology bundle associated to $(F, \nabla^F)$.
\begin{thm}\label{thm 2}
Let $\pi:X\to B$ be a submersion with closed fibers $Z$, equipped with a Riemannian structure $\bm{\pi}$. Let $(F, \nabla^F)$ be a complex flat vector bundle of rank $\ell$ over $X$ and $g^F$ a Hermitian metric on $F\to X$. Then there exist sufficiently large $k, N\in\N$, a morphism $\alpha:k\FF^u\to\CC^{k\ell}$ and a $\Z_2$-graded morphism
$$h:k\HH(Z, F|_Z)\oplus\wh{\CC}^N\to\CC^{km}\oplus\wh{\CC}^N$$
such that
\begin{equation}\label{eq 3.1}
\CS\big(k\nabla^{H(Z, F|_Z), u}\oplus\wh{d}^N, h^*(d^{km}\oplus\wh{d}^N)\big)\equiv\int_{X/B}e(\nabla^{T^VX})\wedge\CS(k\nabla^{F, u}, \alpha^*d^{k\ell}).
\end{equation}
Here, $\CC^{km}=(\C^{km}, g^{km}, d^{km})$ is a $\Z_2$-graded geometric bundle, where $(\C^{km})^\pm=\C^{km_\pm}$ with $m_+:=\rk(H^{\even}(Z, F|_Z))$ and $m_-:=\rk(H^{\odd}(Z, F|_Z)$.
\end{thm}
\begin{proof}
We carry out the proof into several steps.

{\bf Step 1.} Let $k_0\in\N$ satisfy $k_0F\cong k_0\C^\ell$. Let $\alpha:k_0F\to k_0\C^\ell$ be an isomorphism. Let $\nabla^{k_0F}_0=k_0\nabla^F$ and define $\nabla^{k_0F}_1:=\alpha^*d^{k_0\ell}$. Denote by
$$(H_0(Z, k_0F|_Z), \nabla^{H_0(Z, k_0F|_Z)}),\qquad(H_1(Z, k_0F|_Z), \nabla^{H_1(Z, k_0F|_Z)})$$
the cohomology bundles of $(k_0F, \nabla^{k_0F}_0)$ and $(k_0F, \nabla^{k_0F}_1)$, respectively.

Write $n=\dim(Z)$. For each $0\leq j\leq n$, write
$$m_{0, j}=\rk(H^j_0(Z, F|_Z)),\qquad m_{1, j}=\rk(H^j_1(Z, F|_Z)).$$
Let $p_j, q_j\in\N$ satisfy
\begin{equation}\label{eq 3.2}
p_jH^j_0(Z, k_0F|_Z)\cong p_jk_0\C^{m_{0, j}},\qquad q_jH^j_1(Z, k_0F|_Z)\cong q_jk_0\C^{m_{1, j}},
\end{equation}
respectively. For $i\in\set{0, 1}$, write
$$m_{i, +}=\sum_{j\textrm{ even}}m_{i, j},\qquad m_{i, -}=\sum_{j\textrm{ odd}}m_{i, j}\qquad\textrm{ and }\qquad m_i=\sum_{k=0}^nm_{i, j}.$$
Let $p$ be the least common multiple of $p_0, \ldots, p_n, q_0, \ldots, q_n$. By (\ref{eq 3.2}), we have
\begin{equation}\label{eq 3.3}
pH^j_0(Z, k_0F|_Z)\cong pk_0\C^{m_{0, j}},\qquad pH^j_1(Z, k_0F|_Z)\cong pk_0\C^{m_{1, j}}
\end{equation}
for each $0\leq j\leq n$. Let $k=pk_0$. We still denote by $\alpha:kF\to k\C^\ell$ the resulting isomorphism. Note that $\nabla^{kF}_0=k\nabla^F$ and $\nabla^{kF}_1=\alpha^*d^{k\ell}$. By direct summing (\ref{eq 3.3}) over $0\leq j\leq n$, we have
\begin{equation}\label{eq 3.4}
H_0(Z, kF|_Z)\cong\C^{km_0},\qquad H_1(Z, kF|_Z)\cong\C^{km_1}.
\end{equation}
In particular, (\ref{eq 3.4}) implies that
\begin{align}
H^{\even}_0(Z, kF|_Z)&\cong\C^{km_{0, +}},\qquad H^{\even}_1(Z, kF|_Z)\cong\C^{km_{1, +}},\label{eq 3.5}\\
H^{\odd}_0(Z, kF|_Z)&\cong\C^{km_{0, -}},\qquad H^{\odd}_1(Z, kF|_Z)\cong\C^{km_{1, +}}.\label{eq 3.6}
\end{align}

{\bf Step 2.} Write $\EE^u_0=(kF, g^{kF}, \nabla^{kF, u}_0)$ and $\EE^u_1=(kF, g^{kF}, \nabla^{kF, u}_1)$. Let $j\in\set{0, 1}$. Denote by $\DD^{\Lambda\otimes(kF)}_j$ the generalized twisted de Rham operator associated to $\EE^u_j$ and $\bm{\pi}$. Define $\omega_j(kF, g^{kF})$ as in (\ref{eq 2.1}), i.e.
$$\omega_j(kF, g^{kF})=(g^{kF})^{-1}(\nabla^{kF}_jg^{kF}).$$
Let \dis{V_j=-\frac{1}{2}\sum_{i=1}^n\wh{c}(e_i)\omega_j(kF, g^{kF})(e_i)}. Denote by $\DD^{Z, \dr}_j$ the twisted de Rham--Hodge operator associated to $\EE^u_j$ and $\bm{\pi}$ \cite[Definition 3.8]{BL95}. By \cite[(3.36)]{BL95} (see also \cite[Proposition 4.12]{BZ92}), we have
\begin{equation}\label{eq 3.7}
\DD^{\Lambda\otimes(kF); V_j}_j=\DD^{Z, \dr}_j.
\end{equation}
By (\ref{eq 3.7}) and \cite[(3.66)]{BL95}, the kernel bundle $\ker\big(\DD^{\Lambda\otimes(kF); V_j}_j\big)\to B$ exists and
\begin{equation}\label{eq 3.8}
\ker\big(\DD^{\Lambda\otimes(kF); V_j}_j\big)=\ker(\DD^{Z, \dr}_j)\cong H_j(Z, kF|_Z).
\end{equation}

{\bf Step 3.}
Denote by $\wt{\bm{\pi}}$ the Riemannian structure on $\wt{\pi}:\wt{X}\to\wt{B}$ obtained by pulling back $\bm{\pi}$ via $p_X$. Define a Hermitian bundle $(\f, g^\f)$ over $\wt{X}$ by $\f=p_X^*(kF)$ and $g^\f=p_X^*(g^{kF})$. Define a connection $\nabla^\f$ on $\f\to\wt{X}$ by
$$\nabla^\f=dt\wedge\frac{\partial}{\partial t}+(1-t)\nabla^{kF}_0+t\nabla^{kF}_1.$$
Then $\nabla^{\f, u}$ is a unitary connection on $(\f, g^\f)$. By the definitions of $\nabla^{kF}_0$ and $\nabla^{kF}_1$ and (\ref{eq 2.2}), we have
\begin{equation}\label{eq 3.9}
\begin{split}
i_{X, 0}^*\nabla^{\f, u}&=\nabla^{kF, u}_0=k\nabla^{F, u},\\
i_{X, 1}^*\nabla^{\f, u}&=\nabla^{kF, u}_1=\nabla^{kF}_1=\alpha^*d^{k\ell}.
\end{split}
\end{equation}
Write $\bm{\f}^u=(\f, g^\f, \nabla^{\f, u})$. Denote by $\DD^{\Lambda\otimes\f}$ the generalized twisted de Rham operator associated to $\bm{\f}^u$ and $\wt{\bm{\pi}}$. Define $\wt{V}$ by (\ref{eq 2.5}). By applying \cite[Lemma 2.3]{MF79} to $\DD^{\Lambda\otimes\f; \wt{V}}$, let $\bm{\LLL}=(\LLL, g^\LLL, \nabla^\LLL)$ be a $\Z_2$-graded geometric bundle for which $\LLL\to\wt{B}$ satisfies the MF property for $\DD^{\Lambda\otimes\f; \wt{V}}$, i.e. there exists a $\Z_2$-graded complimentary closed subbundle $\KKK\to\wt{B}$ of $\wt{\pi}^\Lambda_*\f\to\wt{B}$ such that
\begin{equation}\label{eq 3.10}
\wt{\pi}^\Lambda_*\f=\KKK\oplus\LLL,
\end{equation}
$\DD^{\Lambda\otimes\f; \wt{V}}_+:(\wt{\pi}^\Lambda_*\f)^+\to(\wt{\pi}^\Lambda_*\f)^-$ is block diagonal as a map with respect to (\ref{eq 3.10}), and the restriction $\DD^{\Lambda\otimes\f; \wt{V}}_+|_{\KKK^+}:\KKK^+\to\KKK^-$ is an isomorphism. Note that
\begin{equation}\label{eq 3.11}
i_{B, j}^*(\wt{\pi}^\Lambda_*\f)^\pm\cong\pi^\Lambda_*(kF)^\pm
\end{equation}
and
$$i_{B, 0}^*\LLL^\pm\cong i_{B, 1}^*\LLL^\pm,\qquad i_{B, 0}^*\KKK^\pm\cong i_{B, 1}^*\KKK^\pm.$$
Define $L_j^\pm=i_{B, j}^*\LLL^\pm$ and $K_j^\pm=i_{B, j}^*\KKK^\pm$. Note that
\begin{equation}\label{eq 3.12}
L_0\cong L_1
\end{equation}
as $\Z_2$-graded complex vector bundles. By (\ref{eq 3.10}) and (\ref{eq 3.11}), we have
\begin{equation}\label{eq 3.13}
\pi^\Lambda_*(kF)^+=K_j^+\oplus L_j^+,\qquad\pi^\Lambda_*(kF)^-=K_j^-\oplus L_j^-,
\end{equation}
and
\begin{equation}\label{eq 3.14}
\DD^{\Lambda\otimes\f; \wt{V}}|_{i_{B, j}^*(\wt{\pi}^\Lambda_*\f)}=\DD^{\Lambda\otimes(kF); V_j}_j.
\end{equation}
It follows from (\ref{eq 3.13}) and (\ref{eq 3.14}) that $\DD^{\Lambda\otimes(kF); V_j}_{j, +}:\pi^\Lambda_*(kF)^+\to\pi^\Lambda_*(kF)^-$ is block diagonal as a map with respect to (\ref{eq 3.13}) and the restriction $\DD^{\Lambda\otimes(kF); V_j}_{j, +}|_{K_j^+}:K_j^+\to K_j^-$ is an isomorphism. Thus $L_j\to B$ satisfies the MF property for $\DD^{\Lambda\otimes(kF); V_j}_j$. Since $\ker\big(\DD^{\Lambda\otimes(kF); V_j}_j\big)\to B$ also satisfies the MF property for $\DD^{\Lambda\otimes(kF); V_j}_j$, it follows from Remark \ref{remark 2} and (\ref{eq 3.8}) that there exist complex vector bundles $V_j\to B$ and $W_j\to B$ such that
$$H_0(Z, kF|_Z)\oplus\wh{W}_0\cong L_0\oplus\wh{V}_0,\qquad H_1(Z, kF|_Z)\oplus\wh{W}_1\cong L_1\oplus\wh{V}_1$$
as $\Z_2$-graded complex vector bundles. Write $L\to B$ for $L_0\to B$. By (\ref{eq 3.12}), we have
\begin{displaymath}
\begin{split}
H_0(Z, kF|_Z)\oplus\wh{W}_0\oplus\wh{V}_1&\cong L\oplus\wh{V}_0\oplus\wh{V}_1\\
&\cong H_1(Z, kF|_Z)\oplus\wh{W}_1\oplus\wh{V}_0.
\end{split}
\end{displaymath}
Rename $\wh{W}_0\oplus\wh{V}_1$ by $\wh{W}_0$ and $\wh{W}_1\oplus\wh{V}_0$ by $\wh{W}_1$. Write $H\to B$ for $V_0\oplus V_1\to B$. Then
\begin{equation}\label{eq 3.15}
H_0(Z, kF|_Z)\oplus\wh{W}_0\cong L\oplus\wh{H}\cong H_1(Z, kF|_Z)\oplus\wh{W}_1
\end{equation}
as $\Z_2$-graded complex vector bundles.

The even part of (\ref{eq 3.15}) is given by
\begin{equation}\label{eq 3.16}
H^{\even}_0(Z, kF|_Z)\oplus W_0\cong H^{\even}_1(Z, kF|_Z)\oplus W_1.
\end{equation}
Since $B$ is compact, there exists a complex vector bundle $V\to B$ such that $W_1\oplus V\cong\C^s$ for some $s\in\N$. By direct summing $V\to B$ and $q\in\N$ copies of $\C^s\to B$ to both sides of (\ref{eq 3.16}) and noting (\ref{eq 3.5}), we have
\begin{equation}\label{eq 3.17}
\C^{km_{0, +}}\oplus\wt{W}_0\cong\C^{km_{1, +}}\oplus\C^{(q+1)s},
\end{equation}
where $\wt{W}_0:=W_0\oplus V\oplus\C^{qs}$. By (\ref{eq 3.17}), $\rk(\wt{W}_0)=(q+1)s+k(m_{1, +}-m_{0, +})$. Choose and fix a sufficiently large $q_0\in\N$ so that
$$(q_0+1)s+k(m_{1, +}-m_{0, +})\geq 2\dim(B).$$
Write $N=(q_0+1)s+k(m_{1, +}-m_{0, +})$. By \cite[Theorem 1.5 of Chapter 9]{H94} (see also \cite[Remark 1]{H20}), we have
\begin{equation}\label{eq 3.18}
\wt{W}_0\cong\C^N.
\end{equation}
Since $H^{\even}_0(Z, kF|_Z)\cong\C^{km_{0, +}}$ (\ref{eq 3.5}), it follows from (\ref{eq 3.18}) that
\begin{equation}\label{eq 3.19}
H^{\even}_0(Z, kF|_Z)\oplus\C^N\cong\C^{km_{0, +}}\oplus\C^N.
\end{equation}
The odd part of (\ref{eq 3.15}) is given by
$$H^{\odd}_0(Z, kF|_Z)\oplus W_0\cong H^{\odd}_1(Z, kF|_Z)\oplus W_1.$$
By the above argument and $H^{\odd}_0(Z, kF|_Z)\cong\C^{km_{0, -}}$ (\ref{eq 3.6}), it becomes
\begin{equation}\label{eq 3.20}
H^{\odd}_0(Z, kF|_Z)\oplus\C^N\cong\C^{km_{0, -}}\oplus\C^N.
\end{equation}
Let $0\leq j\leq n$. Since
$$(H^j_0(Z, kF|_Z), \nabla^{H^j_0(Z, kF|_Z)})\cong k(H^j_0(Z, F|_Z), \nabla^{H^j_0(Z, F|_Z)})$$
as complex flat vector bundles, by dropping the $0$ in $H^j_0(Z, F|_Z)$, $m_{0, j}$, $m_{0, \pm}$ and $m_0$, it follows from (\ref{eq 3.19}) and (\ref{eq 3.20}) that
\begin{equation}\label{eq 3.21}
kH(Z, F|_Z)\oplus\wh{\C}^N\cong\C^{km}\oplus\wh{\C}^N
\end{equation}
as $\Z_2$-graded complex vector bundles.

{\bf Step 4.}
Define a complex vector bundle $J\to B$ by $J=H\oplus V\oplus\C^{q_0s}$. By (\ref{eq 3.15}) and (\ref{eq 3.21}),
\begin{equation}\label{eq 3.22}
L\oplus\wh{J}\cong kH(Z, F|_Z)\oplus\wh{\C}^N\cong\C^{km}\oplus\wh{\C}^N
\end{equation}
as $\Z_2$-graded complex vector bundles. Put Hermitian metrics $g^H$, $g^V$, $g^{km_+}$, $g^{km_-}$, $g^N$ and $g^{q_0s}$ and unitary connections $\nabla^H$, $\nabla^V$, $d^{km_+}$, $d^{km_-}$, $d^N$ and $d^{q_0s}$ on the corresponding vector bundles. Write $\JJ=(J, g^J, \nabla^J)$. Then $\wh{\bm{\JJJ}}:=p_B^*\wh{\JJ}$ is a balanced $\Z_2$-graded geometric bundle over $\wt{B}$.

By (\ref{eq 3.22}), let $h:k\HH(Z, F|_Z)\oplus\wh{\CC}^N\to\CC^{km}\oplus\wh{\CC}^N$ be a $\Z_2$-graded morphism. Note that
\begin{equation}\label{eq 3.23}
k\nabla^{H(Z, F|_Z), u}\oplus\wh{d}^N,\qquad h^*(d^{km}\oplus\wh{d}^N)
\end{equation}
are unitary connections on the Hermitian bundle
$$(kH(Z, F|_Z)\oplus\wh{\C}^N, kg^{H(Z, F|_Z)}\oplus\wh{g}^N).$$
Define a $\Z_2$-graded geometric bundle $\bm{\RRR}$ by $\RRR=p_B^*(kH(Z, F|_Z)\oplus\wh{\C}^N)$, $g^\RRR=p_B^*(kg^{H(Z, F|_Z)}\oplus\wh{g}^N)$ and $\nabla^\RRR$ is defined by \cite[(2.3), (2.4)]{H23} with respect to (\ref{eq 3.23}). Thus
\begin{equation}\label{eq 3.24}
i_{B, 0}^*\nabla^\RRR=k\nabla^{H(Z, F|_Z), u}\oplus\wh{d}^N,\qquad i_{B, 1}^*\nabla^\RRR=h^*(d^{km}\oplus\wh{d}^N).
\end{equation}
Since $i_{B, 0}\circ p_B:\wt{B}\to\wt{B}$ is smoothly homotopic to $\id_{\wt{B}}$, it follows from (\ref{eq 3.22}) that
$$\LLL\oplus\wh{\JJJ}\cong p_B^*(L\oplus\wh{J})\cong p_B^*(kH(Z, F|_Z)\oplus\wh{\C}^N)=\RRR$$
as $\Z_2$-graded complex vector bundles. Let $\varphi:\bm{\LLL}\oplus\wh{\bm{\JJJ}}\to\bm{\RRR}$ be a $\Z_2$-graded morphism. Then
\begin{equation}\label{eq 3.25}
\wt{\nabla}^{\LLL\oplus\wh{\JJJ}}:=\varphi^*\nabla^\RRR
\end{equation}
is a unitary connection on the Hermitian bundle $(\LLL\oplus\wh{\JJJ}, g^\LLL\oplus g^{\wh{\JJJ}})$. Write $\wt{\bm{\LLL}}=(\LLL\oplus\wh{\JJJ}, g^\LLL\oplus g^{\wh{\JJJ}}, \wt{\nabla}^{\LLL\oplus\wh{\JJJ}})$. Since the $\Z_2$-graded morphism $\varphi:\wt{\bm{\LLL}}\to\bm{\RRR}$ satisfies (\ref{eq 3.25}), we identify $\wt{\bm{\LLL}}$ with $\bm{\RRR}$.

{\bf Step 5.} Define a rescaled Bismut superconnection on $\wt{\pi}^\Lambda_*\f\oplus(\LLL\oplus\wh{\JJJ})^{\op}\to\wt{B}$ by
\begin{equation}\label{eq 3.26}
\wh{\bbb}^{\Lambda, \f; \wt{V}}_t=\sqrt{t}\wt{\DD}^{\Lambda\otimes\f; \wt{V}}(\alpha(t))+\big(\nabla^{\wt{\pi}^\Lambda_*\f, u}\oplus\wt{\nabla}^{\LLL\oplus\wh{\JJJ}, \op}\big)-\frac{c(\wt{T})}{4\sqrt{t}}.
\end{equation}
As in (\ref{eq 2.10}) and (\ref{eq 2.12}), we have
\begin{equation}\label{eq 3.27}
\lim_{t\to\infty}\ch(\wh{\bbb}^{\Lambda, \f; \wt{V}}_t)=0
\end{equation}
and by (\ref{eq 3.25}), we have
\begin{equation}\label{eq 3.28}
\begin{split}
\lim_{t\to 0}\ch(\wh{\bbb}^{\Lambda, \f; \wt{V}}_t)&=\int_{\wt{X}/\wt{B}}e(\nabla^{T^V\wt{X}})\wedge\ch(\nabla^{\f, u})-\ch(\wt{\nabla}^{\LLL\oplus\wh{\JJJ}})\\
&=\int_{\wt{X}/\wt{B}}e(\nabla^{T^V\wt{X}})\wedge\ch(\nabla^{\f, u})-\ch(\nabla^\RRR).
\end{split}
\end{equation}
Let $\wh{\eta}^\Lambda(\bm{\f}^u, \wt{\bm{\pi}}, \wt{\bm{\LLL}})$ be the Bismut--Cheeger eta form associated to $\wh{\bbb}^{\Lambda, \f; \wt{V}}_t$ (\ref{eq 3.26}) as in (\ref{eq 2.13}). By (\ref{eq 3.27}) and (\ref{eq 3.28}), we have
\begin{equation}\label{eq 3.29}
d\wh{\eta}^\Lambda(\bm{\f}^u, \wt{\bm{\pi}}, \wt{\bm{\LLL}})=\int_{\wt{X}/\wt{B}}e(\nabla^{T^V\wt{X}})\wedge\ch(\nabla^{\f, u})-\ch(\nabla^\RRR).
\end{equation}
By \cite[(2.1)]{H23}, (\ref{eq 3.29}) and (\ref{eq 3.24}), we have
\begin{equation}\label{eq 3.30}
\begin{split}
&i_{B, 1}^*\wh{\eta}^\Lambda(\bm{\f}^u, \wt{\bm{\pi}}, \wt{\bm{\LLL}})-i_{B, 0}^*\wh{\eta}^\Lambda(\bm{\f}^u, \wt{\bm{\pi}}, \wt{\bm{\LLL}})\\
&\equiv-\int_{\wt{B}/B}d^{\wt{B}}\wh{\eta}^\Lambda(\bm{\f}^u, \wt{\bm{\pi}}, \wt{\bm{\LLL}})\\
&\equiv-\int_{\wt{B}/B}\bigg(\int_{\wt{X}/\wt{B}}e(\nabla^{T^V\wt{X}})\wedge\ch(\nabla^{\f, u})-\ch(\nabla^\RRR)\bigg)\\
&\equiv-\int_{\wt{B}/B}\int_{\wt{X}/\wt{B}}e(\nabla^{T^V\wt{X}})\wedge\ch(\nabla^{\f, u})-\CS\big(k\nabla^{H(Z, F|_Z), u}\oplus\wh{d}^N, h^*(d^{km}\oplus\wh{d}^N)\big).
\end{split}
\end{equation}
Since $e(\nabla^{T^V\wt{X}})=p_X^*e(\nabla^{T^VX})$, it follows from (\ref{eq 3.9}) that
\begin{displaymath}
\begin{split}
-\int_{\wt{B}/B}\int_{\wt{X}/\wt{B}}e(\nabla^{T^V\wt{X}})\wedge\ch(\nabla^{\f, u})&\equiv-\int_{X/B}\int_{\wt{X}/X}p_X^*e(\nabla^{T^VX})\wedge\ch(\nabla^{\f, u})\\
&\equiv-\int_{X/B}e(\nabla^{T^VX})\wedge\int_{\wt{X}/X}\ch(\nabla^{\f, u})\\
&\equiv\int_{X/B}e(\nabla^{T^VX})\wedge\CS(k\nabla^{F, u}, \alpha^*d^{k\ell}).
\end{split}
\end{displaymath}
Thus (\ref{eq 3.30}) becomes
\begin{equation}\label{eq 3.31}
\begin{split}
&i_{B, 1}^*\wh{\eta}^\Lambda(\bm{\f}^u, \wt{\bm{\pi}}, \wt{\bm{\LLL}})-i_{B, 0}^*\wh{\eta}^\Lambda(\bm{\f}^u, \wt{\bm{\pi}}, \wt{\bm{\LLL}})\\
&\equiv\int_{X/B}e(\nabla^{T^VX})\wedge\CS(k\nabla^{F, u}, \alpha^*d^{k\ell})-\CS\big(k\nabla^{H(Z, F|_Z), u}\oplus\wh{d}^N, h^*(d^{km}\oplus\wh{d}^N)\big).
\end{split}
\end{equation}

{\bf Step 6.} To prove (\ref{eq 3.1}), by (\ref{eq 3.31}) it suffices to show that
\begin{equation}\label{eq 3.32}
i_{B, j}^*\wh{\eta}^\Lambda(\bm{\f}^u, \wt{\bm{\pi}}, \wt{\bm{\LLL}})\equiv 0
\end{equation}
for $j\in\set{0, 1}$. We prove (\ref{eq 3.32}) for $j=0$. The argument for $j=1$ is similar.

Let $\psi:H(Z, kF|_Z)\to\ker(\DD^{Z, \dr}_0)$ be an isomorphism given by \cite[(3.66)]{BL95}. Recall that $g^{H(Z, kF|_Z)}:=\psi^*g^{\ker(\DD^{Z, \dr}_0)}$. Write $\KER(\DD^{Z, \dr}_0)$ for the $\Z$-graded geometric bundle $\big(\ker(\DD^{Z, \dr}_0), g^{\ker(\DD^{Z, \dr}_0)}, \nabla^{\ker(\DD^{Z, \dr}_0)}\big)$, where
$$\nabla^{\ker(\DD^{Z, \dr}_0)}:=P^{\ker(\DD^{Z, \dr}_0)}\nabla^{\pi^\Lambda_*(kF), u}.$$
Then $\psi:\HH(Z, kF|_Z)\to\KER(\DD^{Z, \dr}_0)$ is a $\Z$-graded morphism. Since
$$\nabla^{H(Z, kF|_Z), u}=\psi^*\nabla^{\ker(\DD^{Z, \dr}_0)}$$
by \cite[(2.35)]{H20}, we identify $\HH(Z, kF|_Z)$ with $\KER(\DD^{Z, \dr}_0)$. By the identifications of $\wt{\bm{\LLL}}$ with $\bm{\RRR}$ and of $\HH(Z, kF|_Z)$ with $\KER(\DD^{Z, \dr}_0)$, the rescaled Bismut superconnection $\wh{\bbb}^{\Lambda, \f; \wt{V}}_t$ (\ref{eq 3.26}) is equal to
\begin{displaymath}
\begin{split}
\wh{\bbb}^{\Lambda, \f; \wt{V}}_t&=\sqrt{t}\wt{\DD}^{\Lambda\otimes\f; \wt{V}}(\alpha(t))+\big(\nabla^{\wt{\pi}^\Lambda_*\f, u}\oplus\wt{\nabla}^{\LLL\oplus\wh{\JJJ}, \op}\big)-\frac{c(\wt{T})}{4\sqrt{t}}\\
&=\sqrt{t}\wt{\DD}^{\Lambda\otimes\f; \wt{V}}(\alpha(t))+\big(\nabla^{\wt{\pi}^\Lambda_*\f, u}\oplus\nabla^{\RRR, \op}\big)-\frac{c(\wt{T})}{4\sqrt{t}}.
\end{split}
\end{displaymath}
Thus
\begin{equation}\label{eq 3.33}
\begin{split}
i_{B, 0}^*\wh{\bbb}^{\Lambda, \f; \wt{V}}_t&=\sqrt{t}\wt{\DD}^{\Lambda\otimes(kF); V_0}_0(\alpha(t))+\big(\nabla^{\pi^\Lambda_*(kF), u}\oplus(k\nabla^{H(Z, F|_Z), u}\oplus\wh{d}^N)^{\op}\big)-\frac{c(T)}{4\sqrt{t}}\\
&=\bigg(\sqrt{t}\wt{\DD}^{Z, \dr}_0(\alpha(t))+\big(\nabla^{\pi^\Lambda_*(kF), u}\oplus
\nabla^{\ker(\DD^{Z, \dr}_0), \op}\big)-\frac{c(T)}{4\sqrt{t}}\bigg)\oplus\wh{d}^N.
\end{split}
\end{equation}
Note that \dis{\wh{\bbb}^{\dr}_t:=\sqrt{t}\wt{\DD}^{Z, \dr}_0(\alpha(t))+\big(\nabla^{\pi^\Lambda_*(kF), u}\oplus
\nabla^{\ker(\DD^{Z, \dr}_0), \op}\big)-\frac{c(T)}{4\sqrt{t}}} is a rescaled Bismut superconnection on $\pi^\Lambda_*(kF)\oplus\ker(\DD^{Z, \dr}_0)^{\op}\to B$. Write (\ref{eq 3.33}) as
\begin{equation}\label{eq 3.34}
i_{B, 0}^*\wh{\bbb}^{\Lambda, \f; \wt{V}}_t=\wh{\bbb}^{\dr}_t\oplus\wh{d}^N.
\end{equation}
Let $t<T\in(0, \infty)$. By (\ref{eq 3.34}), we have
\begin{equation}\label{eq 3.35}
\begin{split}
i_{B, 0}^*\CS\big(\wh{\bbb}^{\Lambda, \f; \wt{V}}_t, \wh{\bbb}^{\Lambda, \f; \wt{V}}_T\big)&\equiv\CS\big(i_{B, 0}^*\wh{\bbb}^{\Lambda, \f; \wt{V}}_t, i_{B, 0}^*\wh{\bbb}^{\Lambda, \f; \wt{V}}_T\big)\\
&\equiv\CS(\wh{\bbb}^{\dr}_t\oplus\wh{d}^N, \wh{\bbb}^{\dr}_T\oplus\wh{d}^N)\\
&\equiv\CS(\wh{\bbb}^{\dr}_t, \wh{\bbb}^{\dr}_T).
\end{split}
\end{equation}
Let $\bbb^{\dr}_t$ be the Bismut superconnection defined by \cite[(3.49)]{B05}. Then $\wh{\bbb}^{\dr}_t$ is associated to $\bbb^{\dr}_t$ as in (\ref{eq 2.9}). Now suppose $t$ satisfies $t<a$. As in (\ref{eq 2.11}), we have
\begin{equation}\label{eq 3.36}
\wh{\bbb}^{\dr}_t=\bbb^{\dr}_t\oplus\nabla^{\ker(\DD^{Z, \dr}_0), \op}.
\end{equation}
By \cite[(2.8), (2.6), (2.7)]{H23} and (\ref{eq 3.36}), we have
\begin{equation}\label{eq 3.37}
\begin{split}
&\CS(\wh{\bbb}^{\dr}_t, \wh{\bbb}^{\dr}_T)-\CS(\bbb^{\dr}_t, \bbb^{\dr}_T)\\
&\equiv\CS(\wh{\bbb}^{\dr}_t, \wh{\bbb}^{\dr}_T)-\CS(\bbb^{\dr}_t, \bbb^{\dr}_T)-\CS(\nabla^{\ker(\DD^{Z, \dr}_0), \op}, \nabla^{\ker(\DD^{Z, \dr}_0), \op})\\
&\equiv\CS(\wh{\bbb}^{\dr}_t, \wh{\bbb}^{\dr}_T)-\CS(\bbb^{\dr}_t\oplus\nabla^{\ker(\DD^{Z, \dr}_0), \op}, \bbb^{\dr}_T\oplus\nabla^{\ker(\DD^{Z, \dr}_0), \op})\\
&\equiv-\CS(\wh{\bbb}^{\dr}_T, \wh{\bbb}^{\dr}_t)-\CS(\bbb^{\dr}_t\oplus\nabla^{\ker(\DD^{Z, \dr}_0), \op}, \bbb^{\dr}_T\oplus\nabla^{\ker(\DD^{Z, \dr}_0), \op})\\
&\equiv-\CS(\wh{\bbb}^{\dr}_T, \bbb^{\dr}_T\oplus\nabla^{\ker(\DD^{Z, \dr}_0), \op}).
\end{split}
\end{equation}
Denote by $\wh{\eta}^{\dr}$ and $\wt{\eta}^{\dr}$ the Bismut--Cheeger eta forms associated to $\wh{\bbb}^{\dr}_t$ and to $\bbb^{\dr}_t$ as in (\ref{eq 2.13}), respectively. By letting $t\to 0$ and $T\to\infty$ in (\ref{eq 3.37}), we have
$$\wh{\eta}^{\dr}-\wt{\eta}^{\dr}=-\lim_{T\to\infty}\CS(\wh{\bbb}^{\dr}_T, \bbb^{\dr}_T\oplus\nabla^{\ker(\DD^{Z, \dr}_0), \op}).$$
By the estimates in \cite[\S9.3]{BGV}, we have
$$\lim_{T\to\infty}\CS(\wh{\bbb}^{\dr}_T, \bbb^{\dr}_T\oplus\nabla^{\ker(\DD^{Z, \dr}_0), \op})=0.$$
On the other hand, by \cite[Theorem 3.7]{B05} (see also Zhang \cite[Proposition 2.3]{Z04}), $\wt{\eta}^{\dr}=0$. Thus $\wh{\eta}^{\dr}\equiv 0$. By letting $t\to 0$ and $T\to\infty$ in (\ref{eq 3.35}), we have
$$i_{B, 0}^*\CS\big(\wh{\bbb}^{\Lambda, \f; \wt{V}}_t, \wh{\bbb}^{\Lambda, \f; \wt{V}}_T\big)=\wh{\eta}^{\dr}\equiv 0.$$
Thus (\ref{eq 3.32}) holds for $j=0$.
\end{proof}

Recall that the real part of the Cheeger--Chern--Simons class of a complex flat vector bundle $(F, \nabla^F)$ is defined as follows \cite[(2.44)]{MZ08}. Suppose $\rk(F)=\ell$. Let $g^F$ be a Hermitian metric on $F\to X$. Let $k\in\N$ satisfy $kF\cong k\C^\ell$. Define
\begin{equation}\label{eq 3.38}
\re(\CCS(F, \nabla^F))=\bigg[\frac{1}{k}\CS(\alpha^*d^{k\ell}, k\nabla^{F, u})\bigg]\mod\Q,
\end{equation}
where $\alpha:k\FF^u\to\CC^{k\ell}$ is a morphism. If $(F, \nabla^F)$ and $g^F$ are $\Z_2$-graded and $\rk(F^+)=\rk(F^-)$, then $\re(\CCS(F, \nabla^F))$ is defined to be
\begin{equation}\label{eq 3.39}
\re(\CCS(F, \nabla^F))=\bigg[\frac{1}{k}\CS(\alpha^*(k\nabla^{F, u, -}), k\nabla^{F, u, +})\bigg]\mod\Q,
\end{equation}
where $k\in\N$ satisfies $kF^+\cong kF^-$ and $\alpha:k\FF^{+, u}\to k\FF^{-, u}$ is a morphism. In both of these cases, $\re(\CCS(F, \nabla^F))$ does not depend on the choices of $k$ and $\alpha$ (see \cite[\S2.2]{H20}).

\subsection{The real part of the RRG for complex flat vector bundles}\label{s 3.2}

In this subsection, we deduce (\ref{eq 1.2}) from Theorem \ref{thm 2}.
\begin{thm}\label{thm 3}
Let $\pi:X\to B$ be a submersion with closed fibers $Z$, equipped with a Riemannian structure $\bm{\pi}$. For any complex flat vector bundle $(F, \nabla^F)$ over $X$ and a Hermitian metric $g^F$ on $F\to X$, we have
$$\re(\CCS(H(Z, F|_Z), \nabla^{H(Z, F|_Z)}))=\int_{X/B}e(T^VX)\cup\re(\CCS(F, \nabla^F))$$
in $H^{\odd}(B; \R/\Q)$.
\end{thm}
\begin{proof}
Write $\ell=\rk(F)$. By Theorem \ref{thm 2}, there exist sufficiently large $k, N\in\N$, a morphism $\alpha:k\FF^u\to\CC^{k\ell}$ and a $\Z_2$-graded morphism
$$h:k\HH(Z, F|_Z)\oplus\wh{\CC}^N\to\CC^{km}\oplus\wh{\CC}^N$$
such that
\begin{equation}\label{eq 3.40}
\CS\big(k\nabla^{H(Z, F|_Z), u}\oplus\wh{d}^N, h^*(d^{km}\oplus\wh{d}^N)\big)\equiv\int_{X/B}e(\nabla^{T^VX})\wedge\CS(k\nabla^{F, u}, \alpha^*d^{k\ell}).
\end{equation}

By (\ref{eq 3.3}) and (\ref{eq 3.4}), $kH(Z, F|_Z)\cong\C^{km}$ as $\Z$-graded complex vector bundles. Let
$$\wt{h}:k\HH(Z, F|_Z)\to\CC^{km}$$
be a $\Z$-graded morphism. Thus $(\wt{h}\oplus\id_{\wh{\CC}^N})^*(d^{km}\oplus\wh{d}^N)$ is a unitary connection on $kH(Z, F|_Z)\oplus\wh{\C}^N\to B$. Define $f:=h\circ(\wt{h}^{-1}\oplus\id_{\wh{\CC}^N})\in\Aut(\wh{\CC}^N\oplus\CC^{km})$. By \cite[(2.7), (2.8)]{H23}, the left-hand side of (\ref{eq 3.40}) is equal to
\begin{equation}\label{eq 3.41}
\begin{split}
&\CS\big(k\nabla^{H(Z, F|_Z), u}\oplus\wh{d}^N, h^*(d^{km}\oplus\wh{d}^N)\big)\\
&\equiv\CS\big(k\nabla^{H(Z, F|_Z), u}\oplus\wh{d}^N, (\wt{h}\oplus\id_{\wh{\CC}^N})^*(d^{km}\oplus\wh{d}^N)\big)\\
&\qquad+\CS\big((\wt{h}\oplus\id_{\wh{\CC}^N})^*(d^{km}\oplus\wh{d}^N), h^*(d^{km}\oplus\wh{d}^N)\big)\\
&\equiv\CS\big(k\nabla^{H(Z, F|_Z), u}, \wt{h}^*d^{km}\big)+\CS(\wh{d}^N, \wh{d}^N)+\CS\big(d^{km}\oplus\wh{d}^N, f^*(d^{km}\oplus\wh{d}^N)\big)\\
&\equiv\CS\big(k\nabla^{H(Z, F|_Z), u}, \wt{h}^*d^{km}\big)+\CS\big(d^{km}\oplus\wh{d}^N, f^*(d^{km}\oplus\wh{d}^N)\big).
\end{split}
\end{equation}
By the definition of odd Chern character form \cite[Definition 1.1]{G93}, we have
\begin{equation}\label{eq 3.42}
\ch^{\odd}(f, d^{km}\oplus\wh{d}^N)=\CS\big(d^{km}\oplus\wh{d}^N, f^*(d^{km}\oplus\wh{d}^N)\big).
\end{equation}
By (\ref{eq 3.41}), (\ref{eq 3.42}) and \cite[(2.6)]{H23}, (\ref{eq 3.40}) becomes
\begin{equation}\label{eq 3.43}
\begin{split}
&\CS\big(\wt{h}^*d^{km}, k\nabla^{H(Z, F|_Z), u}\big)-\ch^{\odd}(f, d^{km}\oplus\wh{d}^N)\\
&\equiv\int_{X/B}e(\nabla^{T^VX})\wedge\CS(\alpha^*d^{k\ell}, k\nabla^{F, u}).
\end{split}
\end{equation}
Since $\CS(\alpha^*d^{k\ell}, k\nabla^{F, u})$ and $\CS\big(\wt{h}^*d^{km}, k\nabla^{H(Z, F|_Z), u}\big)$ are differential form representatives of $\re(\CCS(F, \nabla^F))$ and $\re(\CCS(H(Z, F|_Z), \nabla^{H(Z, F|_Z)}))$ by (\ref{eq 3.38}) and (\ref{eq 3.39}), respectively, by taking the de Rham class of (\ref{eq 3.43}) and then modding it out by $\Q$, the result follows.
\end{proof}

Since the morphism $\wt{h}$ in the proof of Theorem \ref{thm 3} is $\Z$-graded, it follows that
\begin{equation}\label{eq 3.44}
\CS\big(\wt{h}^*d^{km}, k\nabla^{H(Z, F|_Z), u}\big)=\sum_{j=0}^n(-1)^j\CS\big(\wt{h}_j^*d^{km_j}, k\nabla^{H^j(Z, F|_Z), u}\big).
\end{equation}
By (\ref{eq 3.38}), the right-hand side of (\ref{eq 3.44}) is a differential form representative of
$$\sum_{j=0}^n(-1)^j\re(\CCS(H^j(Z, F|_Z), \nabla^{H^j(Z, F|_Z)})),$$
which is equal to $\re(\CCS(H(Z, F|_Z), \nabla^{H(Z, F|_Z)}))$ by \cite[Lemma 2]{H20}. Thus Theorem \ref{thm 3} becomes
$$\sum_{j=0}^n(-1)^j\re(\CCS(H^j(Z, F|_Z), \nabla^{H^j(Z, F|_Z)}))=\int_{X/B}e(T^VX)\cup\re(\CCS(F, \nabla^F)),$$
which is \cite[(1.4)]{MZ08}.

\bibliographystyle{amsplain}
\bibliography{MBib}
\end{document}